\documentclass[10pt]{amsart}
\usepackage{amsfonts}
\usepackage{latexcad}
\usepackage{natbib}

\theoremstyle{definition}
\newtheorem{definition}{Definition}[section]
\newtheorem{example}[definition]{Example}
\newtheorem{remark}[definition]{Remark}
\theoremstyle{theorem}
\newtheorem{lemma}[definition]{Lemma}
\newtheorem{theorem}[definition]{Theorem}
\newtheorem{proposition}[definition]{Proposition}
\newtheorem{corollary}[definition]{Corollary}

\newtheorem{problem}[definition]{Problem}

\def\pmod#1{\ (\mathrm{mod}\ #1)}

\title{Powers and Alternative Laws}
\author{Nicholas Ormes}
\author{Petr Vojt\v{e}chovsk\'y}
\address{Department of Mathematics, University of Denver, 2360 S Gaylord St,
Denver, Colorado, 80208, U.S.A.}

\email[Ormes]{normes@math.du.edu}
\email[Vojt\v{e}chovsk\'y]{petr@math.du.edu}

\keywords{alternative laws, alternative groupoid, powers, dynamical system,
alternative loop, two-sided inverse}

\thanks{Both authors supported by the 2004 PROF Grant of the University of Denver}

\subjclass{Primary: 20N02, Secondary: 20N05, 37E99}

\begin{document}

\begin{abstract}
A groupoid is alternative if it satisfies the alternative laws $x(xy)=(xx)y$
and $x(yy)=(xy)y$. These laws induce four partial maps on $\mathbb{N}^+\times
\mathbb{N}^+$
\begin{equation*}
(r,\,s)\mapsto (2r,\,s-r),\quad (r-s,\,2s),\quad (r/2,\,s+r/2),\quad
(r+s/2,\,s/2),
\end{equation*}
that taken together form a dynamical system. We describe the orbits of this
dynamical system, which allows us to show that $n$th powers in a free
alternative groupoid on one generator are well-defined if and only if $n\le 5
$. We then discuss some number theoretical properties of the orbits, and the
existence of alternative loops without two-sided inverses.
\end{abstract}

\maketitle

\section{Alternative laws and the induced dynamical systems}\label{Sc:Dyn}

Let $G$ be a free groupoid with one generator $x$. The elements of $G$ are
(correctly parenthesized) words built from the single letter $x$. The
\emph{length} $|w|$ of a word $w$ is the number of letters in $w$.

For a positive integer $n$ we denote by $x^n$ any of the words of length $n$ in
$G$. Note that there are precisely $c_{n}$ such words, where $c_{n}$ is the
$n$th \emph{Catalan number} defined by the recursive relation $c_{0}=1$,
$c_{1}=1$, $c_{n+1}=c_{1}c_{n}+c_{2}c_{n-1}+\cdots +c_{n-1}c_{2}+c_{n}c_{1}$,
cf. \cite{LintWilson}.

A groupoid is said to be \emph{left alternative} if it satisfies the left
alternative law $x(xy)=(xx)y$. Dually, it is \emph{right alternative} if it
satisfies the right alternative law $x(yy)=(xy)y$. A groupoid that is both
left alternative and right alternative is called \emph{alternative}. (When
dealing with algebras, the \emph{flexible law} $x(yx)=(xy)x$ is counted
among alternative laws, and hence alternative algebras by definition satisfy
the flexible law in addition to the left and right alternative laws. Our
terminology is common for nonassociative structures with one binary
operation.)

Let $A$ be the free alternative groupoid with generator $x$. Then $A$
consists of equivalence classes of $G$, where two elements of $G$ are
equivalent if and only if they can be obtained from each other by finitely
many applications of the alternative laws. For instance, the equivalence
class of $(xx)(xx)$ consists of all possible powers $x^{4}$, as is
immediately seen from $x((xx)x)=x(x(xx))=(xx)(xx)=((xx)x)x=(x(xx))x$ and from
the fact that $c_{4}=5$. Thus the words of the form $x^{4}$ form an
equivalence class in $A$, i.e., $x^{4}$ is \emph{well-defined}.

The goal of this paper is to determine for which $n>0$ the power $x^{n}$ is
well-defined in $A$, and to investigate related questions. In Sections
\ref{Sc:FlipsAlt} and \ref{Sc:Loop} we turn our attention to alternative loops
without two-sided inverses. The following concept proves useful in all of these
tasks:

Consider a word $w=uv$ in $G$ such that $|u|=r>0$, $|v|=s>0$. Assume that $w$
is transformed into $w^{\prime}=u^{\prime}v^{\prime}$ by a single application
of an alternative law. Then $(|u^{\prime}|,|v^{\prime}|)$ is either $(r,s)$
(when the law is applied inside $u$ or $v$), or $(2r,s-r)$ (when $w=u(ut)$), or
$(r-s,2s)$ (when $w=(tv)v$), or $(r/2,s+r/2)$ (when $w=(tt)v$), or
$(r+s/2,s/2)$ (when $w=u(tt)$). This suggests the introduction of these partial
maps on $\mathbb{N}^+\times \mathbb{N}^+$:
\begin{align*}
&\alpha(r,s)=(2r,s-r),\quad & \beta(r,s)&=(r-s,2s), \\
&\gamma(r,s)=(r/2,s+r/2),\quad & \delta(r,s)&=(r+s/2,s/2).
\end{align*}
Note that $\alpha$ is defined if and only if $s>r$, $\beta$ is defined if and
only if $r>s$, $\gamma$ is defined if and only if $r$ is even, and $\delta$
is defined if and only if $s$ is even. Also note that $\alpha$ is the inverse
of $\gamma$, and $\beta$ is the inverse of $\delta$ (in the sense that
$\alpha\circ\gamma$, $\gamma\circ\alpha$, $\beta\circ\delta$,
$\delta\circ\beta$ are identity maps on their respective domains).

It is not a novel idea to think of partial maps on integers as a dynamical
system---the most notorious example being the dynamical system on
$\mathbb{N}^{+}$ associated with the $3n+1$ problem \cite{3kp1}. In that case
there are two maps
\begin{eqnarray*}
\mu \left( r\right) &=&\frac{r}{2}, \\
\nu \left( r\right) &=&\frac{3r+1}{2},
\end{eqnarray*}
$\mu$ is defined for even $r$, $\nu$ for odd $r$, and the (open) problem is
whether $1$ can be found in the orbit of every $r$.

Just as in the $3n+1$ problem, we are interested in the orbits of the
dynamical system. We define the \emph{orbit} of $(r,s)\in
\mathbb{N}^{+}\times \mathbb{N}^{+}$ as the set
\begin{equation*}
O(r,s)=\{\varphi _{k}\varphi _{k-1}\cdots \varphi _{1}(r,s);\;k\geq
0,\,\varphi _{i}\in \{\alpha ,\,\beta ,\,\gamma ,\,\delta \}\}.
\end{equation*}

Note that although the dynamical system is defined on $\mathbb{N}^{+}\times
\mathbb{N}^{+}$, it is really a union of one-dimensional dynamical systems,
since $a+b=r+s$ for every $(a,b)\in O(r,s)$.

The orbits do not capture the equivalence classes of $G$, of course, but they
provide some information about them. In particular, if
$u_{1}v_{1}=u_{2}v_{2}$ in $A$, then $(|u_{1}|,|v_{1}|)\in
O(|u_{2}|,|v_{2}|)$.

More information about $A$ can be recovered by considering higher dimensional
dynamical systems. For an integer $m>1$, let $T$ denote all binary trees with
$m$ leaves. Let $t\in T$ be one of the trees and $w=u_{1}\cdots u_{m}$ a word
bracketed according to $t$, with $|u_{i}|=r_{i}$. Then the alternative laws
apply to $w$ and produce words bracketed according to some $t^{\prime }\in T$
with subwords of some lengths $r_{1}^{\prime }$, $\dots $, $r_{m}^{\prime }$.
The \emph{orbit} of $(r_{1},\dots ,r_{m})\in (\mathbb{N} ^{+})^{m}$ then
consists of all $m$-tuples $(r_{1}^{\prime },\dots,r_{m}^{\prime })$ obtained
as above from all trees $t\in T$ and all words $w $ bracketed according to
$t$ with subwords of lengths $r_{1}$, $\dots $, $r_{m}$.

In full generality, the structures that describe the action of identities on
terms are known as geometry monoids, with which one can associate so-called
syntactical monoids. See \cite{Dehornoy1, Dehornoy2, Dehornoy3}.

When $m=2$, we do not have to worry about all possible bracketings, since
the two top factors are uniquely specified in a given word of $G$. Since we
will need the dynamical systems of dimension $m>2$ only on one occasion
(Lemma \ref{Lm:Dim3}), we do not discuss them here any further.

\section{Orbits}

\label{Sc:Orbits}

We are now going to describe the general shape of any orbit $O(r,s)$. The key
observation is the following:

When $r<s$ then $\alpha $ applies to $(r,s)$ and
\begin{equation*}
\alpha (r,s)=(2r,s-r)=(2r\ \mathrm{mod}\ (r+s),\,2s\ \mathrm{mod}\ (r+s)).
\end{equation*}
When $r>s$ then $\beta $ applies to $(r,s)$ and
\begin{equation*}
\beta (r,s)=(r-s,2s)=(2r\ \mathrm{mod}\ (r+s),\,2s\ \mathrm{mod}\ (r+s)).
\end{equation*}
Hence the two partial maps $\alpha $, $\beta $ can be replaced by a single
partial map $\omega $ on $\mathbb{N}^{+}\times \mathbb{N}^{+}$ given by
\begin{equation*}
\omega (r,s)=(2r\ \mathrm{mod}\ (r+s),\ 2s\ \mathrm{mod}\ (r+s)),
\end{equation*}
defined if and only if $r\neq s$. Moreover, since $\alpha $ is the inverse of
$\gamma $, and $\beta $ is the inverse of $\delta$, $\omega $ is the left
inverse of both $\gamma $ and $\delta $. (This peculiarity arises because
$\gamma$, $\delta $ are not defined everywhere.)

\setlength{\unitlength}{1.28mm}
\begin{figure}[th]
\begin{center}
\input{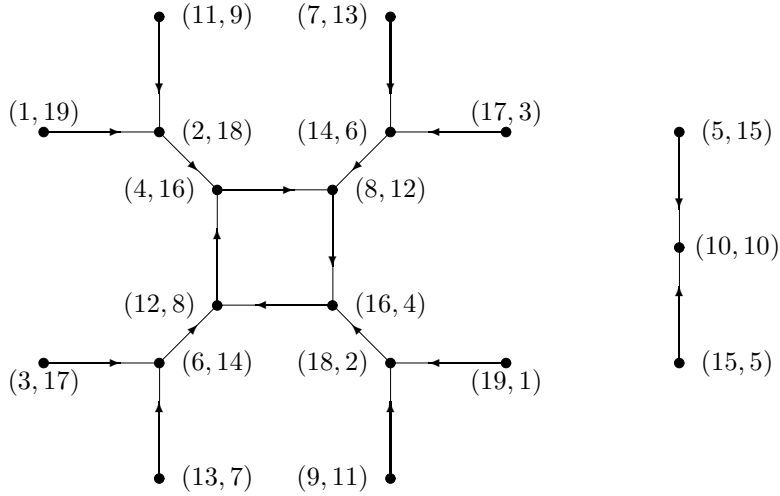}
\end{center}
\caption{Orbits for $n=r+s=20$.} \label{Fg:Orb20}
\end{figure}

Figure \ref{Fg:Orb20} shows all (two) orbits $O(r,s)$ with $r+s=20$. The
arrows in the figure stand for a single application of $\omega $. All
features of orbits that we are going to discuss are already displayed in the
figure.

\begin{lemma}\label{Lm:Omega}
Let $a$, $b$, $r$, $s>0$. Then $(a,b)\in O(r,s)$ if and only if there are
$m$, $n\in \mathbb N$ such that $\omega ^{m}(a,b)=\omega ^{n}(r,s)$.
\end{lemma}
\begin{proof}
Every map $\varphi_i$ in the definition of $O(r,s)$ is either $\omega$, or a
right inverse of $\omega$. On any path from $(r,s)$ to $(a,b)$ in the orbit, a
right inverse of $\omega$ cannot be followed by $\omega$, of course. Hence
there is at most one point along the path where $\omega$ is followed by a right
inverse of $\omega$.
\end{proof}

Let $O$ be an orbit and $g>0$ an integer. We say that a point $v=(r,s)\in O$
\emph{has depth} $g$ (or \emph{is at depth} $g$) if $\gcd (r,s)=g$. We also
denote the depth of $v$ by $\gcd (v)$, and define $\gcd (V)=\max_{v\in V}\gcd
(v)$ for any subset $V$ of $O$. In particular, $\gcd (O)$ is the \emph{depth}
of $O$. The subset of $O$ consisting of all $v\in O$ of maximum depth will be
called the \emph{bottom} of $O$, and denoted by $B=B(O)$.

\begin{lemma}\label{Lm:g2g}
Let $r$, $s>0$, $r\neq s$, $r+s=n=2^{a}b$, $b$ odd. Let $g=\gcd
(r,s)=2^{c}d$, $d$ odd, and $(r',s')=\omega (r,s)$. Then $g'=\gcd (r',s')\in
\{g,2g\}$, and $g'=g$ if and only if $a=c$.
\end{lemma}
\begin{proof}
Without loss of generality, let $r'=2r$, $s'=s-r$ (i.e., $\omega=\alpha$).
Since $g$ is a common divisor of $r$, $s$, it is also a common divisor of
$r'$, $s'$, and thus $g|g'$. It then follows that $g'\in\{g,2g\}$, because
$\gcd(r,s)=\gcd(r,s')$ and $r'=2r$.

When $a=c$ then $g'\ne 2g$ (equivalently, $g'=g$) because $g'$ is a divisor
of $n=r'+s'$ but $2g$ is not. When $a>c$ then $g'=2g$ because both $r'=2r$
and $s'=n-r'$ are divisible by $2g$.
\end{proof}

Hence the depth increases by the factor of $2$ with every application of
$\omega$ until the bottom is reached.

\begin{corollary}\label{Cr:DiffGcd}
 Let $r$, $s>0$, $(a,b)\in O(r,s)$. Then $\gcd(a,b)/\gcd(r,s) = 2^m$ for some
 $m\in\mathbb Z$.
\end{corollary}

\begin{lemma}\label{Lm:Preimages}
Let $r$, $s>0$, $r+s=n$, and $g=\gcd(r,s)$. Then exactly one of the following
is true:
\begin{enumerate}
\item[(i)] $g$ is odd, $n$ is even, and $(r,s)$ has no $\omega $-preimages,

\item[(ii)] both $g$, $n$ are odd, $(r,s)$ has one $\omega $-preimage
$(r',s')$, and $\gcd(r',s')=g$,

\item[(iii)] $g$ is even and $(r,s)$ has two $\omega $-preimages $(r',s')$,
$(r'',s'')$.
\end{enumerate}
Furthermore, if \emph{(iii)} holds and $g'=\gcd(r',s')=\gcd(r'',s'')$ then
$g=2g'$.
\end{lemma}
\begin{proof}
First note that the case $g$ even, $n$ odd cannot occur.

An $\omega $-preimage of $(r,s)$ has two possible forms:
$\gamma(r,s)=(r/2,s+r/2)=(r/2,(s+n)/2)$, or
$\delta(r,s)=(r+s/2,s/2)=((r+n)/2,s/2)$.

When $g$ is odd and $n$ is even then both $r$ and $s$ are odd, so neither
$r/2$ nor $s/2$ are integers, and $(r,s)$ has no $\omega $-preimages.

When $g$ and $n$ are both odd then one of $r$, $s$ is even and the other is
odd. Without loss of generality assume that $r$ is even and $s$ is odd. Then
$r'=r/2$, $s'=(s+n)/2$ are integers, and $(r',s')$ is the unique
$\omega$-preimage of $(r,s)$. Moreover,
$g=\gcd(r,s)=\gcd(r,s+n)=\gcd(r/2,(s+n)/2)$ since $g$ is odd.

Assume that $g$ is even. Then $n$, $r$, $s$ are all even. Hence
$(r',s')=(r/2,(s+n)/2)$, $(r'',s'')=((r+n)/2,s/2)$ are $\omega$-preimages of
$(r,s)$.

Assume further that $g'=\gcd(r',s')=\gcd(r'',s'')$. Then $g'|(r/2)$,
$g'|(s/2)$, and $g=2g'$ follows by Lemma \ref{Lm:g2g}.
\end{proof}

\begin{proposition}
\label{Pr:Bottom} The bottom $B(O)$ of an orbit $O$ is either a single point
or a directed cycle. Moreover, $B(O)$ is a singleton if and only if $(r,r)\in
O$ for some $r$, in which case $B(O)=\{(r,r)\}$.
\end{proposition}
\begin{proof}
Let $B=B(O)$. If $(r,r)\in O$ then $B=\{(r,r)\}$ because $r=\gcd (r,r)>\gcd
(a,b)$ for any $(a,b)\in O$ with $a\neq b$ (since either $a<r$ or $b<r$).

Assume that $(r,r)\not\in O$. Let $v\in B$. Then Lemma \ref{Lm:g2g} implies
$\omega (v)\in B$. By induction, $C=\{\omega ^{k}(v)$; $k\geq 0\}\subseteq
B$. By finiteness, there is a least $t>0$ such that $\omega ^{t}(v)=\omega
^{k}(v)$ for some $0\leq k<t$. If $k>0$ then $\omega ^{k}\left( v\right) $
has two distinct $\omega $-preimages, $\omega ^{k-1}\left( v\right) $ and
$\omega ^{t-1}\left( v\right) $, both at the same depth, which contradicts
Lemma \ref{Lm:Preimages}. Hence $k=0$ and $C$ is a directed cycle.

Suppose there is $v'\in B\setminus C$, and let $C'=\{\omega^k(v');\;k\ge 0\}$
be the directed cycle determined by $v'$. By Lemma \ref{Lm:Omega}, there are
$m$, $n\in\mathbb N$ such that $\omega^m(v)=\omega^n(v')=v''$. Then $v''\in
C\cap C'$, and $v'\in C$ follows, a contradiction.
\end{proof}

Let us call a rooted tree $T$ an \emph{extended complete binary tree} if $T$
is rooted at $r$ and consists of an edge $rt$ and a complete binary tree
attached to $t$. The \emph{height} of a rooted tree is the length of its
longest branch.

\begin{theorem}[Shape of orbits]
\label{Th:ShapeO} Every orbit $O=O(r,s)$ consists of a directed cycle
$B=B(O)$, possibly degenerated into a point, to which disjoint trees are
attached. If $|B|>1$, there is one tree attached to every node of $B$. If
$|B|=1$, there are two trees attached to the unique node of $B$. Moreover:

\begin{enumerate}
\item[(i)] when $|O|>1$, each tree attached to $B$ is an extended complete
binary tree of height $a$, where $\gcd (B)=2^{a}b$, $b$ odd,

\item[(ii)] $O=B=\{(2^{k}r\ \mathrm{mod}\ (r+s),\,2^{k}s\ \mathrm{mod}\
(r+s));\;k\geq 0\}$ if and only if $r+s$ is odd,

\item[(iii)] $|O|=|B|2^{a}$ if $|B|>1$, and $|O|=2^{a+1}-1$ if $|B|=1$.
\end{enumerate}
\end{theorem}

\begin{proof}
We use Lemmas \ref{Lm:g2g}, \ref{Lm:Preimages} and Proposition
\ref{Pr:Bottom} freely. Suppose that $(r,s)$ is a node at even depth $g$.
Then $(r,s)$ has two $\omega $-preimages. If $(r,s)\in B$ and $|B|>1$ then
exactly one of these $\omega $-preimages is in the cycle $B$, while the other
is at depth $g/2$. If $(r,s)\not\in B$, both $\omega $-preimages of $(r,s)$
are at depth $g/2$. The binary tree arising in this process keeps growing
until the shallowest depth $b$ is reached, where $\gcd (B)=2^{a}b$, $b$ odd.
The rest follows from the fact that a complete binary tree of height $h$ has
$1+2+\cdots +2^{h}=2^{h+1}-1$ nodes.
\end{proof}

Note that Theorem \ref{Th:ShapeO} implies that the shape of the orbit
$O(r,s)$ is determined once the length of the bottom cycle and the highest
power of $2$ dividing $r+s$ are known.

In the following lemma we let $B(r,s)$ to denote the bottom of $O(r,s)$.

\begin{lemma}
\label{Lm:OrbGcd} Let $r$, $s>0$, and let $t$ be a common divisor of $r$,
$s$. Assume $(r,s)\in B(r,s)$. Then $(r/t,s/t)\in B(r/t,s/t)$ and
$B(r,s)=t\cdot B(r/t,s/t)=\{(ta,tb) ;\;(a,b) \in B(r,s)\}$. When $g=\gcd
(r,s)$, then $r/g+s/g$ is odd and $B(r,s)=g\cdot B(r/g,s/g)=g\cdot
O(r/g,s/g)$.
\end{lemma}
\begin{proof}
Let $r+s=2^{a}b$, $b$ odd, $g=\gcd (r,s)=2^{c}d$, $d$ odd. Then $(r,s)\in
B(r,s)$ holds if and only if $a=c$, by Lemma \ref{Lm:g2g}. Hence $(r,s)\in
B(r,s)$ implies $(r/t,s/t)\in B(r/t,s/t)$. We leave the verification of the
equality $B(r,s)=t\cdot B(r/t,s/t)$ to the reader. Finally, $r/g+s/g$ is odd
since $g=2^ad$, $d$ odd, and we are done by Theorem \ref{Th:ShapeO}(ii).
\end{proof}

\section{Complete Orbits and Well-defined Powers}

\label{Sc:Powers}

An orbit $O(r,s)$ is said to be \emph{complete} if $|O(r,s)|=r+s-1$, that is
$O\left( r,s\right) $ contains all pairs $(a,b)$ with $a$, $b>0$ and
$a+b=r+s$. An integer $n$ is said to be \emph{complete} if any (and hence
all) orbits $O(r,s)$ with $r+s=n$ are complete.

Recalling the introduction, we say that $x^n$ in the free alternative
groupoid generated by $x$ is \emph{well-defined}, if the expression $x^n$ is
independent of parentheses. Obviously, if $x^n$ is well-defined, then $n$
must be complete.

For a prime $p$, let $\mathrm{GF}(p)$ be the field of order $p$. Recall that
$a\in \mathrm{GF}(p)$ is a \emph{primitive element} of $\mathrm{GF}(p)$ if it
generates the multiplicative group $\mathrm{GF}(p)\setminus\{0\}$.

\begin{proposition}
\label{Pr:CompleteOrbits} An integer $n>0$ is complete if and only if either
$n=2^m$ or $n=p$ is an odd prime and $2$ is a primitive element of
$\mathrm{GF}(p)$.
\end{proposition}
\begin{proof}
Assume that $n=2^m$. Consider $(1,n-1)\in O(1,n-1)=O$. Since
$\omega^{m-1}(1,n-1)=(2^{m-1},2^{m-1})$, we have
$B(O)=\{(2^{m-1},2^{m-1})\}$. Then $|O|=2^m-1$ by Theorem
\ref{Th:ShapeO}(iii), showing that $n$ is complete.

If $n=p$ is an odd prime then $n$ is complete if and only if $2$ is a
primitive element of $\mathrm{GF}(p)$, by Theorem \ref{Th:ShapeO}(ii).

Assume that $n$ is not an odd prime and $n\neq 2^{m}$. Let $p$ be an odd
prime (properly) dividing $n$. Then $\gcd(1,n-1)=1$ and $\gcd(p,n-p)=p$, and
thus $(1,n-1)$, $(p,n-p)$ cannot be in the same orbit, by Corollary
\ref{Cr:DiffGcd}. This means that $n$ is not complete.
\end{proof}

\begin{lemma}\label{Lm:3Orbits}
Suppose that $n>pq$ is divisible by $pq$, where $p$, $q$
are odd primes, not necessarily distinct. Then there are at least three
distinct orbits $O(r,s)$ with $r+s=n$.
\end{lemma}
\begin{proof}
Use the elements $(1,n-1)$, $(p,n-p)$ and $(pq,n-pq)$ with Corollary
\ref{Cr:DiffGcd}.
\end{proof}

\begin{lemma}\label{Lm:5}
If $1\le n\le 5$, the power $x^n$ is well-defined in the free alternative
groupoid generated by $x$.
\end{lemma}

\begin{proof}
There is nothing to prove for $n\le 2$. Any one of the alternative laws shows
that $x^3$ is well-defined. We have shown in the introduction that $x^4 $ is
well-defined.

Since $x^m$ are well-defined for every $m<5$, to prove that $x^5$ is
well-defined it suffices to show that for every $1<i<5$ some word $uv$ with
$|u|=i$ can be obtained from $x(x(x(xx)))$. Now, $x(x(x(xx))=(xx)(x(xx)) =
(xx)((xx)x) = ((xx)(xx))x = (((xx)x)x)x = ((xx)x)(xx)$ does just that.
\end{proof}

Because $6$ is not complete by Proposition \ref{Pr:CompleteOrbits}, $x^6$ is
not well-defined. However, we cannot conclude right away that $x^n$ is not
well-defined for every $n>5$. The catch is that it could happen that the
alternative laws apply to higher powers, say $x^8$, in so many ways that
$x^n$ could be well-defined. The following technical lemma will help us
eliminate such a possibility:

\begin{lemma}
\label{Lm:CompDef} Let $u$, $v$ be words such that $|u|+|v|=n$ and $|u|$ is
odd. Further assume that one of the following two conditions holds:
\begin{enumerate}
\item[(i)] $n$ is odd and not complete,

\item[(ii)] $n$ is even, $n$ is as in Lemma $\ref{Lm:3Orbits}$, $n/2$ is odd.
\end{enumerate}
Then there is a word $w$ of length $n$ such that $u(uv)\ne uw$ in the free
alternating groupoid $A$ generated by $x$. In particular, $x^{n+|u|}$ is not
well-defined in $A$.
\end{lemma}
\begin{proof}
Let $|u|=r$, $|v|=s$, $r+s=n$. Since $n$ is not complete, there exists
$0<t<n$ such that $(r,s)\not\in O(t,n-t)$. Should we assume (ii), we can
further demand that $t\ne n/2$, by Lemma \ref{Lm:3Orbits}.

Let $w$ be any word of the form $w_1w_2$ where $|w_1|=t$, $|w_2|=n-t$. If
$u(uv)=u(w_1w_2)$, then any proof of this fact must involve some of the
letters in the left-most $u$, because $(r,s)\not\in O(t,n-t)$, and thus
$uv\ne w_1w_2$. Since such a proof terminates in $u(w_1w_2)$, there is a step
in the proof when the word becomes $u_1u_2$, $|u_1|=r$, and such that all
additional steps are performed inside $u_1$ or $u_2$. We claim that such a
step is either impossible, or yields $u_2$ that cannot be transformed to
$w_1w_2$.

How could the word $u_1u_2$ be produced? Assume it is produced by the left
alternative law. Then the step is either $y(yu_2)=(yy)u_2=u_1u_2$,
contradicting $|u|=|u_1|$ odd, or it is $(u_1u_1)z = u_1(u_1z)=u_1u_2$, in
which case $u_2$ cannot be transformed to $w_1w_2$ because $(r,s)\not\in
O(t,n-t)$. Now assume that the step in question is produced by the right
alternative law. When the step is $y(u_2u_2)=(yu_2)u_2=u_1u_2$ we reach a
contradiction as $|u_1|=r<|u_2|=n$. Assume the step is $(u_1z)z = u_1(zz) =
u_1u_2$. This is clearly impossible when $n$ is odd. Otherwise (ii) is
assumed, and hence $t=|w_1|\ne n/2$. But $u_2=zz$, $|z|=n/2 $, $n/2$ is odd
and $|O(n/2,n/2)|=1$, so $u_2$ cannot be transformed into $w_1w_2$.
\end{proof}

\begin{lemma}
The power $x^n$ is well-defined in the free alternative groupoid generated by
$x$ if and only if $n\le 5$, except possibly $n=11$, $n=13$.
\end{lemma}
\begin{proof}
Because $7$ is not complete, Lemma \ref{Lm:CompDef} implies that $x^8$ is not
well-defined. Let $n=2^m>8$, and let $k$ be the largest odd multiple of $3$
smaller than $n$. Note that either $n=k+1$ or $n=k+5$, and thus $k>3$,
$n<2k$. Since $k$ is not complete by Proposition \ref{Pr:CompleteOrbits},
$x^n $ is not well-defined by Lemma \ref{Lm:CompDef}. Any even $n$ that is
not a power of $2$ is not complete, by Proposition \ref{Pr:CompleteOrbits},
and we have therefore shown that $x^n$ is not well-defined for any even
$n>4$.

Now assume that $n>5$ is odd. All odd multiples of $18$ satisfy assumption
(ii) of Lemma \ref{Lm:CompDef}. The Lemma therefore implies that $x^n$ is not
well-defined for any odd $n$ between $18$ and $36$, between $3\cdot 18=54$
and $108$, between $5\cdot 18=90$ and $180$, etc. Moreover, $30$ satisfies
assumption (ii) of Lemma \ref{Lm:CompDef}, too, and thus $x^n$ is not
well-defined for any odd $n$ between $30$ and $60$. Since none of $7$, $9$,
$15$, $17$ is complete, by Proposition \ref{Pr:CompleteOrbits}, we are
through.
\end{proof}

A more subtle argument (a higher dimensional dynamical system) is needed to
eliminate the possibility that $x^{11}$ or $x^{13}$ is well-defined:

\begin{lemma}
\label{Lm:Dim3} Let $A$ be the free alternative groupoid generated by $x$.
Then $x^3x^5\ne x^4x^4$ in $A$. Consequently, $x^{11}$, $x^{13}$ are not
well-defined in $A$.
\end{lemma}
\begin{proof}
Since $3$, $5$ are odd, we can only transform $x^3x^5$ into $(x^3x^3)x^2$.
Because $|O(3,3)|=1$, we can either return to $x^3(x^3x^2)$, or proceed to
$((x^3x^3)x)x$, from which point we cannot proceed any further. The powers we
have reached are of the form $x^3x^5$, $x^6x^2$, $x^7x$, but not $x^4x^4$.

Note that $(x^3x^3)x^5$ can only be transformed to $x^3(x^3x^5)$. Since
$x^3x^5\ne x^4x^4$, the expression $x^3(x^3x^5)$ can only be transformed to
$x^3x^8$ or to $(x^3x^3)x^5$. Thus $x^{11}$ is not well-defined.

Similarly, $x^3(x^5x^5)$ can only be transformed to $(x^3x^5)x^5$ (because
$|O(5,5)|=1$). Since $x^3x^5\ne x^4x^4$, we conclude that $x^{13}$ is not
well-defined either.
\end{proof}

\begin{corollary}
The power $x^n$ is well-defined in a free alternative groupoid generated by
$x$ if and only if $n\le 5$.
\end{corollary}

\section{Flips}\label{Sc:Flips}

For $r$, $s>0$, we say that $(r,s)$ \emph{flips} if $(s,r)\in O(r,s)$.
Whether or not the element $(1,r)$ flips is related to the existence of
two-sided inverses in alternative loops. We explain this in more detail in
the next section. For the time being, we can think of flipping as a concept
related to well-defined powers.

Clearly, if $r+s$ is complete or if $r=s$ then $(r,s)$ flips. We first
investigate flips for $r+s$ odd. The situation is more transparent in the odd
case then in the even case thanks to Theorem \ref{Th:ShapeO}(ii). The even
case is handled by Proposition \ref{Pr:EvenF}.

The suspected connection to number theory reveals itself in the following
result:

\begin{proposition}
\label{Pr:FNT} Assume that $r$, $s>0$ and $r+s$ is odd. Then $(r,s)$ flips if
and only if there is $k\geq 0$ such that $2^{k}r\equiv -r\ (\mathrm{mod}\
r+s)$. If further $\gcd (r,s)=1$, then $(r,s)$ flips if and only if there is
$k\ge 0$ such that $2^{k}\equiv -1\ (\mathrm{mod}\ r+s)$.
\end{proposition}

\begin{proof}
All congruences in this proof are modulo $r+s$. By Theorem
\ref{Th:ShapeO}(ii), $(r,s)$ flips if and only if there is $k\ge 0$ such that
$s\equiv 2^kr$ , $r\equiv 2^ks$. Since $r+s\equiv 2^k(r+s)$, we see that the
above two congruences hold if and only if at least one of them holds, say
$s\equiv 2^kr $. As $s\equiv -r$, $(r,s)$ flips if and only if $2^kr\equiv
-r$.

Assume that $\gcd (r,s)=1$. Then $\gcd (r,r+s)=1$, too, and the last
congruence is therefore equivalent to $2^{k}\equiv -1$.
\end{proof}

When $\gcd(r,s)\ne 1$, the situation can be reduced to the relatively prime
case (see Proposition \ref{Pr:EvenF}). We are thus interested in solutions $k$
to the congruence
\begin{equation}\label{Eq:Congr}
2^{k}\equiv -1\ (\mathrm{mod}\ n).
\end{equation}
Of course, \eqref{Eq:Congr} has no solution when $n$ is even. When $n$ is
odd, the behavior of \eqref{Eq:Congr} appears to be a difficult number
theoretic question, related to the classical problem whether $2$ is a
primitive element modulo $n$. We do not fully understand for which values of
$n$ the congruence \eqref{Eq:Congr} has a solution. Nevertheless, based on
the prime factorization of $n$, we can identify many values of $n$ for which
there is no solution, and others for which there is a solution.

\begin{lemma}\label{Lm:Order}
Suppose $p$ is an odd prime. The congruence $2^{k}\equiv -1 \pmod p$ has a
solution if and only if the multiplicative order of $2$ in $\mathrm{GF}(p)$
is even.
\end{lemma}
\begin{proof}
All congruences in this proof are modulo $p$. Assume that the order of $2$ is
an odd number $m$ and that $2^{k}\equiv -1 \pmod p$. Then we have a
contradiction via $1\equiv 1 ^{k}\equiv \left( 2^{m}\right) ^{k}\equiv \left(
2^{k}\right) ^{m}\equiv \left( -1\right) ^{m}\equiv -1$.

On the other hand, if the order of $2$ is an even number $m$, then we have
$2^{m}-1 \equiv \left( 2^{m/2}-1\right) \left( 2^{m/2}+1\right) \equiv 0$,
which implies $2^{m/2}\equiv 1$ or $2^{m/2}\equiv -1$, the former
contradicting the fact that $2$ is of order $m$.
\end{proof}

Many of our results are based on quadratic residues. We recall some of the
relevant definitions and results from elementary number theory. (See
\cite{Burton}.)

Let $p$ be an odd prime and $\gcd (a,p)=1$. Then $a$ is said to be a
\emph{quadratic residue modulo $p$} if the congruence $x^{2}\equiv a\
(\mathrm{mod}\ p) $ has a solution. The \emph{Legendre symbol} $(a/p)$ is then
defined by
\begin{equation*}
(a/p)=\left\{
\begin{array}{rl}
1, & \text{if $a$ is a quadratic residue of $p$}, \\
-1, & \text{otherwise}.
\end{array}
\right.
\end{equation*}

\begin{lemma}[Euler's Criterion]
Let $p$ be an odd prime and $\gcd (a,p)=1$. Then
\begin{equation*}
(a/p)=a^{(p-1)/2}\ \mathrm{mod}\ p.
\end{equation*}
\end{lemma}

\begin{lemma}
\label{Lm:2p} Assume that $p$ is an odd prime. Then
\begin{equation*}
(2/p)=\left\{
\begin{array}{rl}
1, & \text{if $p\equiv 1\ (\mathrm{mod}\ 8)$ or $p\equiv {7}\ (\mathrm{mod}\
8)$}, \\
-1, & \text{if $p\equiv 3\ (\mathrm{mod}\ 8)$ or $p\equiv {5}\ (\mathrm{mod}
\ 8)$}.
\end{array}
\right.
\end{equation*}
\end{lemma}

In particular, the following proposition now follows easily:

\begin{proposition}\label{Pr:Aux}
Let $p$ be an odd prime, and let $r$, $s>0$ be such that $r+s=p$,
$\gcd(r,s)=1$. Then:
\begin{enumerate}
\item[(i)] if $p\equiv 3$ or $5\pmod 8$ then $(r,s)$ flips,

\item[(ii)] if $p\equiv 7\pmod 8$ then $(r,s)$ does not flip.
\end{enumerate}
\end{proposition}
\begin{proof}
Assume $p\equiv 3$ or $5\pmod 8$. Then $2^{(p-1)/2}\equiv-1\pmod p$, so
\eqref{Eq:Congr} has a solution, and (i) follows.

To see (ii), note that if $p\equiv 7\pmod 8$ then $2^{(p-1)/2}\equiv 1\pmod
p$. Hence the order of $2$ in $\mathrm{GF}(p)$ divides $(p-1)/2$, which is
odd. Then \eqref{Eq:Congr} has no solution by Lemma \ref{Lm:Order}.
\end{proof}

\begin{theorem}[Chinese Remainder Theorem]
\label{Th:Chinese} Let $n_1$, $\dots$, $n_m$ be pairwise relatively prime
integers, and let $a_1$, $\dots$, $a_m$ be integers. Then there exists a
unique solution $x$ of the system of congruences $x\equiv a_i\pmod{n_i}$ with
$0\le x< n_1n_2\cdots n_m$.
\end{theorem}

\begin{lemma}\label{Lm:2power}
Let $n$ be odd, and let $r$, $s>0$ be such that $r+s=n$ and $\gcd (r,s)=1$.
Then the following conditions are equivalent:
\begin{enumerate}
\item[(i)] $(r,s)$ flips,

\item[(ii)] there is $a\in\mathbb N$ such that for every prime $p$ dividing
$n$ there is $k_p$ satisfying $2^{k_p}\equiv -1\pmod p$ such that
$k_p=2^ab_p$, $b_p$ odd,

\item[(iii)] there is $a\in\mathbb N$ such that for every prime $p$ dividing
$n$ there is $k_p$ satisfying $2^{k_p}\equiv -1\pmod p$, and every such
solution satisfies $k_p=2^ab_p$ for some odd $b_p$.
\end{enumerate}
\end{lemma}
\begin{proof}
We first show the equivalence of (i) and (ii).

Assume that (i) holds. Then $2^k\equiv -1\pmod n$ for some $k$. Thus
$2^k\equiv -1\pmod p$ for every prime divisor $p$ of $n$, and it suffices to
set $k_p=k$.

Conversely, assume (ii), let $n=p_1^{m_1}\cdots p_\ell^{m_\ell}$ be the prime
factorization of $n$, and let $k_i=2^ab_i$ be such that $b_i$ is odd and
$2^{k_i}\equiv -1\pmod {p_i}$ for every $i$. We first show by induction on
$r$ that $2^{kp^{r-1}}\equiv -1\pmod {p^r}$ holds for every $k=k_i$, $p=p_i$
and $r>0$. There is nothing to prove for $r=1$. Let $x=2^{kp^{r-1}}$ and
assume that $x\equiv -1\pmod{p^r}$. Since $p$ is odd, we have $x^p+1 =
(x+1)\sum_{i=0}^{p-1}(-1)^ix^i$. By the induction hypothesis, $p^r$ divides
$x+1$. Then $x\equiv -1\pmod p$, too, and thus
$\sum_{i=0}^{p-1}(-1)^ix^i\equiv \sum_{i=0}^{p-1}(-1)^i(-1)^i\equiv p\equiv
0\pmod p$. Altogether, $p^{r+1}$ divides $x^p+1$, and the claim is proved.
Now let $K=2^a(b_1p_1^{m_1-1}\cdots b_\ell p_\ell^{m_\ell -1})$. By the
claim, $2^{k_ip_i^{m_i-1}}\equiv -1\pmod{p_i^{m_i}}$. As $k_i=2^ab_i$, $K$ is
an odd multiple of $k_ip_i^{m_i-1}$, and so $2^K\equiv -1\pmod{p_i^{m_i}}$.
Then $2^K\equiv -1\pmod n$ by the Chinese Remainder Theorem, and $(r,s)$
flips.

Condition (iii) implies (ii). To see the converse, it suffices to show that any
two solutions to $2^k\equiv -1\pmod p$ are divisible by the same powers of $2$.
Assume that $2^u\equiv -1\equiv 2^v\pmod p$, $u=2^ab$, $b$ odd, $v=2^cd$, $d$
odd, and that $a>c$. Then $-1\equiv (-1)^d\equiv (2^u)^d\equiv 2^{2^abd}\equiv
2^{2^cdb2^{a-c}}\equiv (2^v)^{b2^{a-c}}\equiv (-1)^{b2^{a-c}}\equiv 1$, a
contradiction.
\end{proof}

\begin{theorem}\label{Th:Primes}
Let $n$ be odd, and let $r$, $s>0$ be such that $r+s=n$ and $\gcd (r,s)=1$.
Then:
\begin{enumerate}
\item[(i)] if every prime $p$ dividing $n$ satisfies $p\equiv 3\pmod 8$ then
$(r,s)$ flips,

\item[(ii)] if every prime $p$ dividing $n$ satisfies $p\equiv 5\pmod 8$ then
$(r,s)$ flips,

\item[(iii)] if $n$ is divisible by primes $p\equiv 3$, $q\equiv 5\pmod 8$
then $(r,s)$ does not flip,

\item[(iv)] if $n$ is divisible by a prime $p$ with $p\equiv 7\pmod 8$ then
$(r,s)$ does not flip.
\end{enumerate}
\end{theorem}
\begin{proof}
For a prime divisor $p$ of $n$, let $k_p=(p-1)/2$. When $p\equiv 3\pmod 8$
then $k_p$ is odd and $2^{k_p}\equiv -1\pmod p$. When $p\equiv 5\pmod 8$ then
$k_p$ is even, not divisible by $4$, and $2^{k_p}\equiv -1\pmod p$. Parts
(i), (ii) and (iii) therefore follow by Lemma \ref{Lm:2power}. When $p\equiv
7\pmod 8$ then $2^k\equiv -1\pmod p$ has no solution, as explained in the
proof of Proposition \ref{Pr:Aux}, and (iv) follows again by Lemma
\ref{Lm:2power}.
\end{proof}

\begin{remark}
The cases not covered in Theorem \ref{Th:Primes} seem to be complicated. For
instance:
\begin{enumerate}
\item[] $17\equiv 1\pmod 8$ and $(1,16)$ flips,

\item[] $73\equiv 1\pmod 8$ and $(1,72)$ does not flip,

\item[] $51 = 17\cdot 3\equiv 1\cdot 3\pmod 8$, and $(1,50)$ does not flip,

\item[] $843 = 281\cdot 3 \equiv 1\cdot 3\pmod 8$, and $(1,842)$ flips,

\item[] $85 = 17.5\equiv 1\cdot 5\pmod 8$, and $(1,84)$ does not flip,

\item[] $205 = 41\cdot 5 \equiv 1\cdot 5\pmod 8$, and $(1,204)$ flips.
\end{enumerate}
In order to fully understand such situations, we would have to know not only
whether \eqref{Eq:Congr} has a solution $k$ for $n=p\equiv 1\pmod 8$, but
also the highest power of $2$ dividing the solution (and hence all
solutions).
\end{remark}

The following Proposition tells us how to proceed in the even case or when
$\gcd(r,s)\ne 1$:

\begin{proposition}
\label{Pr:EvenF} Let $r$, $s>0$, $r+s=n=2^{a}b$, $b$ odd, $\gcd
(r,s)=2^{c}d$, $d$ odd, $(r',s')=\omega ^{a-c}(r,s)$, $g'=\gcd (r',s')$. Then
$r'/g'+s'/g'$ is odd, and $(r,s)$ flips if and only if $(r'/g',s'/g')$ flips.
\end{proposition}
\begin{proof}
The point $(r',s')$ is the first point at the bottom $B=B(r,s)$ encountered
along the unique directed path from $(r,s)$, by Theorem \ref{Th:ShapeO}. By
Lemma \ref{Lm:OrbGcd}, $B=g'\cdot B(r'/g',s'/g')$ and $r'/g'+s'/g'$ is odd.
The rest follows.
\end{proof}

We conclude this section with an example:

\begin{example}
Does $(r,s)=(435,137)$ flip? Since $435+137=572=2^{2}\cdot 143$ and $\gcd
(435,137)=1$, we look at $(r',s')=\omega^2(r,s) = (4\cdot 435\ \mathrm{mod} \
572,\,4\cdot 147\ \mathrm{mod}\ 572)=(24,548)$. Since $g'=\gcd (24,548)=4$,
we know by Proposition \ref{Pr:EvenF} that $(r,s)$ flips if and only if
$(24/4,548/4)=(6,137)$ flips. The odd sum $6+137=143$ factors as $143=11\cdot
13$. Since $11\equiv 3\ (\mathrm{mod}\ 8)$ and $13\equiv 5\ (\mathrm{mod}\
8)$, Theorem \ref{Th:Primes} tells us that $(6,137)$ does not flip. Hence
$(435,137)$ does not flip.
\end{example}

\section{Flips and alternative loops}\label{Sc:FlipsAlt}

A groupoid with neutral element in which the equation $ab=c$ has a unique
solution whenever two of the element $a$, $b$, $c$ are given is known as a
\emph{loop}. In particular, we can cancel on the left and on the right in a
loop, i.e., $xy=xz$ or $yx=zx$ implies $y=z$. Multiplication tables of finite
loops are therefore precisely normalized Latin squares. See
\cite{Pflugfelder} for an introductory text on the theory of loops.

Let $L$ be a loop with neutral element $e$. Then for every $x\in L$ there
are uniquely determined $y$, $z\in L$ such that $xy=zx=e$. If $y=z$, we say
that $x$ has a \emph{two-sided inverse}.

A loop is \emph{alternative} if it satisfies the left and right alternative
laws. Although it is trivial to construct finite loops without two-sided
inverses, no finite alternative loops without two-sided inverses are known.

\begin{problem}[Warren D.~Smith (2004)]
\label{Pr:AL} Is there a finite alternative loop without two-sided inverses?
\end{problem}

To see the connection between this problem and flips, consider the following:

Assume that $L$ is a finite loop, and let $x\in L$. Define the left powers
$x^{(n)}$ recursively by $x^{(0)}=e$, $x^{(n+1)} = xx^{(n)}$. Let $x^{[n]}$
denote the analogously defined right powers.

By finiteness of $L$, there is a smallest positive integer $n$ such that
$x^{(n)}=x^{(m)}$ for some $0\le m<n$. If $m>0$ then
$xx^{(n-1)}=x^{(n)}=x^{(m)}=xx^{(m-1)}$, and the left cancelation implies that
$x^{(n-1)}=x^{(m-1)}$, a contradiction with the minimality of $n$.

Thus for every $x$ there exists $n$ such that $xx^{(n)}=e$. Similarly, there
exists $m$ such that $x^{[m]}x=e$. Clearly then, $x$ has a two-sided inverse
if and only if $x^{(n)}=x^{[m]}$ for the above integers $n$, $m$.

Assume that $xx^{(n)}=e$. If it were possible to conclude that $x^{(n)}x=e$
by using alternative laws only, then $(1,n)$ would have to flip. However, we
know that $(1,n)$ does not flip for all values of $n$.

Thus any proof of Problem \ref{Pr:AL} must involve either cancelation or the
neutral element $e$. For instance, one could prove $xx^{(n)}=x^{[m]}x$ by
showing $v(xx^{(n)})=v(x^{[m]}x)$ for some word $v$, and then canceling $v$. As
we are going to show in the next section, the finiteness of the loop in
question must also be incorporated into any such proof.

We conclude this section with an existence result of arbitrarily long intervals
$(r,1)$, $\dots$, $(r,s)$ where no $(r,i)$ flips.

\begin{theorem}[Dirichlet]
\label{Th:Dirichlet} Let $a$, $b$ be relatively prime integers. Then the
arithmetic progression $an+b$ contains infinitely many primes.
\end{theorem}

\begin{lemma}
\label{Lm:OrbMod} Let $r$, $s$, $t>0$ be such that $t$ is odd, divides $r+s$,
and does not divide $r$. If there is a directed path in $O(r,s)$ from $(r,s)$
to $(r',s')$, then there is a directed path from $(r\ \mathrm{mod}\ t,\,s\
\mathrm{mod}\ t)$ to $(r'\ \mathrm{mod}\ t,\, s'\ \mathrm{mod}\ t)$ in $O(r\
\mathrm{mod}\ t,\, s\ \mathrm{mod}\ t) $.
\end{lemma}
\begin{proof}
Since $t$ divides $r+s$ and $t$ does not divide $r$, we see that $(r\
\mathrm{mod}\ t) + (s\ \mathrm{mod}\ t) = t$.

The first coordinate of $\omega(r,s)$ is $2r\ \mathrm{mod}\ (r+s)$. The first
coordinate of $\omega(r\ \mathrm{mod}\ t$, $s\ \mathrm{mod}\ t)$ is $[2(r\
\mathrm{mod}\ t)]\ \mathrm{mod}\ (r\ \mathrm{mod}\ t\ +\ s\ \mathrm{mod}\ t)] =
2(r\ \mathrm{mod}\ t)\ \mathrm{mod}\ t = 2r\ \mathrm{mod}\ t$. Thus the first
coordinate of $\omega(r,s)$ is mapped onto the first coordinate of $\omega(r\
\mathrm{mod}\ t,\,s\ \mathrm{mod}\ t)$ under the map $u\mapsto u\ \mathrm{mod}\
t$. Since the sum of coordinates is preserved under $\omega$, an analogous
statement holds for the second coordinate.

Note that $t$ divides $(2r\ \mathrm{mod}\ (r+s)) + (2s\ \mathrm{mod}\
(r+s))=r+s$. Thus, if we show that $t$ does not divide $2r\ \mathrm{mod}\
(r+s)$, we can repeat the step in the previous paragraph as many times as we
wish; hence finishing the proof. Now, $(2r\ \mathrm{mod}\ (r+s))\
\mathrm{mod}\ t = 2r\ \mathrm{mod}\ t\ne 0$ since $t$ is odd and does not
divide $r$.
\end{proof}

\begin{proposition}
\label{Pr:NonFlip} Let $M$ be a positive integer. Then there exists $r>0$
such that none of $(r,1)$, $(r,2)$, $\dots$, $(r,M)$ flips.
\end{proposition}
\begin{proof}
By Theorem \ref{Th:Dirichlet}, there are infinitely many primes congruent to
$7$ modulo $8$. Let $p_1<p_2<\cdots<p_M$ be among such primes, and assume
further that $M<p_1$. Since $p_1$, $p_2$, $\dots$, $p_M$ are pairwise
relatively prime, there is a solution $r>0$ to the system of congruences
$r+s\equiv 0\ (\mathrm{mod}\ p_s)$, $1\le s\le M$. We claim that none of
$(r,1)$, $\dots$, $(r,M)$ flips.

Let $1\leq s\leq M$, $r+s=2^{a}b$, $b$ odd. Since $p_{s}$ is odd and divides
$r+s=2^{a}b$, $p_{s}$ must divide $b$. But $s\leq M<p_{1}\leq p_{s}$, and
hence $b$ cannot divide $s$. Thus $(r\ \mathrm{mod}\ b)+(s\ \mathrm{mod}\
b)=b$. By Theorem \ref{Th:Primes}(iv), $(r\ \mathrm{mod}\ b,s\ \mathrm{mod}\
b)$ does not flip. By Lemma \ref{Lm:OrbMod}, $(r,s)$ does not flip.
\end{proof}

\section{An infinite alternative loop without two-sided
inverses}\label{Sc:Loop}

It is essential to include the word ``finite'' in the statement of Problem
\ref{Pr:AL}, as there are infinite alternative loops without two-sided
inverses. The existence of such a loop was suggested by J.~D.~Phillips, and
it was constructed for the first time by Warren D.~Smith.

When our construction below is used with the parameters $S=\{a_0=1$, $a_1\}$,
it yields Smith's loop. Our contribution should be regarded as a
straightforward generalization of Smith's idea. We split it into several
steps:

Let $S_0$, $S_1$ be cyclic subgroups of an abelian group $S=(S,\cdot,1)$,
with generators $s_0$, $s_1$, respectively. Assume that $s_0\ne s_1$. Let
$x_0$, $x_1$ be symbols. Set $L=\{ax_i^n;\;a\in S$, $i\in\{0,1\}$,
$n\in\mathbb N\}$ and identify $x_i^0$ with $1$. Define multiplication
$\circ$ on $L$ by
\begin{equation}\label{Eq:ALT}
    ax_i^n\circ bx_j^m = \left\{\begin{array}{ll}
        (ab)x_i^{n+m},&i=j,\\ \\
        (abs_i^n)x_j^{m-n},&i\ne j,\,n\le m,\\ \\
        (abs_i^m)x_i^{n-m},&i\ne j,\,n\ge m.
    \end{array}\right.
\end{equation}
Note that the two bottom branches yield the same result when $n=m$, namely
$abs_i^n$.

\begin{lemma} $(L,\circ)$ defined by \eqref{Eq:ALT} satisfies:
\begin{enumerate}
\item[(i)] $L$ is closed under $\circ$,

\item[(ii)] $a\circ b = ab$ for $a$, $b\in S$, and thus $S\le L$,

\item[(iii)] $a\in S$ commutes and associates with every element of $L$.
\end{enumerate}
\end{lemma}
\begin{proof}
Parts (i), (ii) are straightforward. We have $a\circ bx_j^m = (ab)x_j^m =
(ba)x_j^m = bx_j^m\circ a$ for $a$, $b\in S$, $m\in\mathbb N$. To show that
$a\in S$ associates with all elements of $L$, it suffices to show that
$(a\circ bx_i^n)\circ cx_j^m = a\circ (bx_i^n\circ cx_j^m)$ and $(bx_i^n\circ
cx_j^m)\circ a = bx_i^n\circ (cx_j^m\circ a)$ for every $b$, $c\in S$, $n$,
$m\in\mathbb N$. Assume $n\le m$. Then $(a\circ bx_i^n)\circ cx_j^m =
(ab)x_i^n\circ cx_j^m = (abcs_i^n)x_j^{m-n} = a \circ (bcs_i^n)x_j^{m-n} =
a\circ (bx_i^n\circ cx_j^m)$. Assume $n\ge m$. Then $(a\circ bx_i^n)\circ
cx_j^m = (ab)x_i^n\circ cx_j^m = (abcs_i^m)x_i^{n-m} = a\circ
(bcs_i^m)x_i^{n-m} = a\circ (bx_i^n\circ cx_j^m)$. The other equality is
proved similarly.
\end{proof}

\begin{lemma} $(L,\circ)$ is a loop without two-sided inverses.
\end{lemma}
\begin{proof}
We need to show that $x\circ y=z$ has a unique solution in $L$ whenever two
of the elements $x$, $y$, $z$ are given. We prove this when $x$ and $z$ are
given, the other case being analogous.

Let $x=ax_i^n$, $z=bx_j^m$. When $i=j$ and $n\le m$, we have $x\circ
(a^{-1}b)x_i^{m-n} = z$. When $i=j$ and $n\ge m$, we let $k\ne i$ and have
$ax_i^n \circ (a^{-1}bs_i^{-(n-m)})x_k^{n-m} =
(aa^{-1}bs_i^{-(n-m)}s_i^{n-m})x_i^{n-(n-m)} = z$. When $i\ne j$, we have
$ax_i^n \circ (a^{-1}bs_i^{-n})x_j^{m+n} =
(aa^{-1}bs_i^{-n}s_i^n)x_j^{(m+n)-n} = z$. It is not hard to see that the
above solutions are unique in all cases.

Now, $x_0\circ s_0^{-1}x_1 = s_0^{-1}s_0=1$ and $s_1^{-1}x_1\circ x_0 =
s_1^{-1}s_1=1$ together with $s_0\ne s_1$ implies that $x_0\in L$ does not
have a two-sided inverse.
\end{proof}

\begin{theorem} $(L,\circ)$ is an alternative loop without two-sided
inverses.
\end{theorem}
\begin{proof}
It remains to show that the alternative laws hold in $L$. We only prove the
left alternative law, the right alternative law being analogous.

Let $a$, $b\in S$, $n$, $m\in\mathbb N$, $i\ne j$. Then
\begin{displaymath}
    (ax_i^n\circ ax_i^n)\circ bx_j^m = a^2x_i^{2n}\circ bx_j^m = \left\{
    \begin{array}{ll}
        (a^2bs_i^{2n})x_j^{m-2n},& 2n\le m,\\ \\
        (a^2bs_i^m)x_i^{2n-m},& 2n\ge m.
    \end{array}\right.
\end{displaymath}
On the other hand, when $n\le m$ we have
\begin{displaymath}
    ax_i^n\circ (ax_i^n\circ bx_j^m) = ax_i^n\circ(abs_i^n)x_j^{m-n} = \left\{
    \begin{array}{ll}
        (a^2bs_i^ns_i^n)x_j^{m-2n},&2n\le m,\\ \\
        (a^2bs_i^ns_i^{m-n})x_i^{2n-m},&2n\ge m,
    \end{array}\right.
\end{displaymath}
and when $n\ge m$ we have
\begin{displaymath}
    ax_i^n\circ (ax_i^n\circ bx_j^m) = ax_i^n\circ (abs_i^m)x_i^{n-m} =
    (a^2bs_i^m)x_i^{2n-m}.
\end{displaymath}
Careful comparison of cases then shows that the left alternative law holds.

When $i=j$, the left alternative law obviously holds.
\end{proof}

Note that all powers $x^n$ with $n>0$ are well defined in the loop
(\ref{Eq:ALT}).

\section{Acknowledgement}

The second author worked briefly on Problem \ref{Pr:AL} with J.~D.~Phillips and
Warren D.~Smith. Consequent work of Smith resulted in an unpublished manuscript
\cite{Smith}. The maps $\alpha$, $\beta$, $\gamma$, $\delta$ of Section
\ref{Sc:Dyn} are discussed in \cite{Smith}, and results concerning mirrorable
integers and primes are obtained there. (An integer $n$ is \emph{mirrorable} if
the left power $x^{(n)}$ is equal to the right power $x^{[n]}$. Therefore, if
$n$ is mirrorable then $(1,n-1)$ flips, but not necessarily vice versa.) All
results of Sections \ref{Sc:Orbits}, \ref{Sc:Powers} and \ref{Sc:Flips} are
new, to our knowledge.


\bibliographystyle{plain}

\end{document}